\documentclass[11pt,reqno]{amsart}
\usepackage{amssymb,amsfonts,amsthm,amscd,stmaryrd,dsfont,esint,upgreek,constants,todonotes}
\usepackage[mathcal]{euscript}
\usepackage[latin1]{inputenc}   
\usepackage[mathcal]{euscript}
\usepackage{graphicx,color}
\numberwithin{equation}{section}
\newcommand{\N}{\mathbb N}

\newcommand{\Q}{\mathbb Q}
\newcommand{\R}{\mathbb R}
\def\E{\mathbb E}
\def\P{\mathbb P}
\newconstantfamily{C}{symbol=C}
\newconstantfamily{c}{symbol=c}
\newconstantfamily{O}{symbol=\Omega}

\newcommand{\br}{{\bf r}}

\numberwithin{equation}{section}
\newtheorem{thm}{Theorem}[section]
\newtheorem{lem}[thm]{Lemma}
\newtheorem{cor}[thm]{Corollary}
\newtheorem{prop}[thm]{Proposition}

\theoremstyle{definition}

\def\smallnegint{\mathop{\int\mkern-13mu
        \raise.5ex\hbox{${\scriptscriptstyle\diagup}$}}\nolimits}
\def\ds{\displaystyle}

\def\ep{\varepsilon}

\def\F{{\mathcal F}}

\def\tr{\operatorname{tr}}

\def\ssetminus{\,\raise.4ex\hbox{$\scriptstyle\setminus$}\,}

\newcommand{\be}{\begin{equation}}
\newcommand{\ee}{\end{equation}}

\newcommand{\Rd}{\mathbb{R}^d}

\renewcommand{\bar}{\overline}
\renewcommand{\tilde}{\widetilde}
\renewcommand{\hat}{\widehat}

\begin{document}

\title[Correctors in homogenization of HJ equations]{On the existence of correctors for the stochastic homogenization of viscous Hamilton-Jacobi equations}
\author[Pierre Cardaliaguet and Panagiotis E. Souganidis]
{Pierre Cardaliaguet and Panagiotis E. Souganidis}
\address{Ceremade, Universit\'e Paris-Dauphine,
Place du Mar\'echal de Lattre de Tassigny, 75775 Paris cedex 16 - France}
\email{cardaliaguet@ceremade.dauphine.fr }
\address{Department of Mathematics, University of Chicago, Chicago, Illinois 60637, USA}
\email{souganidis@math.uchicago.edu}
\vskip-0.5in 
\thanks{\hskip-0.149in Cardaliaguet was partially supported by the ANR (Agence Nationale de la Recherche) project  ANR-12-BS01-0008-01. Souganidis was partially supported by the National Science
Foundation Grants DMS-1266383 and DMS-1600129 and the Office of Naval Research Grant  N000141712095.}
\dedicatory{Version: \today}


\maketitle      


\section*{abstract}

\noindent We prove, under some assumptions,  the existence of correctors for the stochastic homogenization of of ``viscous'' possibly degenerate Hamilton-Jacobi equations
in stationary ergodic media. The general claim is that, assuming knowledge of homogenization in probability, correctors exist for  all extreme points of the convex hull of the sublevel sets of the effective Hamiltonian. 
Even when homogenization is not a priori  known, the arguments imply existence of correctors and, hence, homogenization in some new settings. These include positively homogeneous Hamiltonians and, hence, geometric-type equations including motion by mean curvature, in radially symmetric environments and for all directions. Correctors also exist and, hence, homogenization holds for many directions for non convex Hamiltonians and general stationary ergodic media.  

\section*{R\'esum\'e}

Nous d\'emontrons l'existence, sous certaines conditions, de correcteurs en homog\'en\'eisation stochastique d'\'equations de Hamilton-Jacobi et d'\'equations de Hamilton-Jacobi visqueuses. L'\'enonc\'e g\'en\'eral est que, si l'on sait qu'il y a homog\'en\'eisation en probabilit\'e, un correcteur existe pour toute  direction \'etant un point extr\'emal de l'enveloppe convexe d'un ensemble de niveau du Hamiltonien effectif. 
M\^{e}me lorsque que l'homog\'en\'eisation n'est pas connue a priori, les arguments d\'evelopp\'es dans cette note montrent l'existence d'un correcteur, et donc l'homog\'enisation, dans certains contextes. Cela inclut les \'equations de type g\'eom\'etrique dans des environnements dont la loi est \`{a} symm\'etrie radiale. Dans le cas g\'en\'eral stationnaire ergodique et sans hypoth\`{e}se de convexit\'e sur le hamiltonien, on montre que des correcteurs existent pour plusieurs directions.

\section{Introduction}
The aim of this note is to show the existence of  correctors for the  stochastic homogenization of ``viscous'' Hamilton-Jacobi equations of the form 
\begin{equation}
\label{e.VHJ}
u^\ep_t -\ep \tr\left(A\left(Du^\ep,\frac x\ep,\omega \right)D^2u^\ep\right)+ H\left(Du^\ep, \frac x\ep,\omega\right) = 0 \quad \mbox{in} \ \  \Rd \times (0,\infty).
\end{equation}
Here $\ep > 0$ is a small parameter which tends to zero, $H=H(p,y,\omega)$ is the Hamiltonian and $A=A(p,y,\omega)$ is a (possibly) degenerate  diffusion matrix. Both $A$ and $H$ depend on a parameter $\omega \in \Omega$, where $(\Omega,\F,\P)$ is a probability space. We assume that $\P$ is stationary ergodic with respect to translations on $\Rd$ and that  $A$ and $H$ are stationary.

\smallskip

The basic question in the stochastic homogenization of~\eqref{e.VHJ} is the existence of a deterministic effective Hamiltonian~$\overline H$ such that the solutions $u^\ep$ to~\eqref{e.VHJ} converge, as $\ep \to 0$, locally uniformly and  with probability one, to the solution to the effective equation
\begin{equation}\label{HJh}
u_t + \overline H(Du) = 0 \quad \mbox{in} \ \  \Rd \times (0,\infty).
\end{equation}
When $H$ is convex with respect to the $p$ variable and coercive, this was first proved independently by Souganidis~\cite{S} and Rezakhanlou and Tarver~\cite{RT} for first order  Hamilton-Jacobi equations, and later extended to the viscous setting by Lions and Souganidis \cite{LS2} and Kosygina, Rezakhanlou and Varadhan \cite{KRV}. See also Armstrong and Souganidis~\cite{ASo1, ASo3} and Armstrong and Tran~\cite{AT2} for generalizations and alternative arguments. \smallskip

In periodic homogenization the convergence and, hence, homogenization rely on the existence of correctors (see Lions, Papanicolaou and Varadhan~\cite{LPV}). The random setting is, however, fundamentaly different. 
\smallskip

Following Lions and Souganidis~\cite{LS1}, a corrector associated with a direction $p\in \R^d$ is a solution $\chi$ to the corrector equation 
\be\label{eq:correctorIntro}
-{\rm tr}(A(D\chi(x)+p,x,\omega)D^2\chi(x)) +H(D\chi(x)+p, x, \omega)=\bar H(p) \quad {\rm in } \ \  \R^d
\ee
which has a sublinear growth at infinity, that is, with probability one,  
\be\label{cond.sublin}
\lim_{|x|\to +\infty} \frac{\chi(x, \omega)}{|x|} = 0. 
\ee

It was shown in \cite{LS1} that in general such solutions do not exist; note that main point is the existence of solutions satisfying \eqref{cond.sublin}. \smallskip

Not knowing how to find correctors is the main reason that the theory of homogenization in random media is rather complicated and required the development of new arguments. General qualitative results in the references cited earlier required the quasiconvexity assumption. A more direct approach to prove homogenization (always in the convex setting), which is based on weak convergence methods and yields only convergence in probablitity, was put forward by Lions and Souganidis~\cite{LS3}. Our approach here is close in spirit to the one of \cite{LS3}.   With the exception of a case with Hamiltonians of a very special form (see Armstrong, Tran and Yu \cite{armstrong2015stochastic, armstrong2016stochastic}), the main results known in nonconvex settings are quantitative. That is it is necessary to make some strong assumptions on the environment (finite range dependence) and to use sophisticated concentration inequalities to prove directly that the solutions of the oscillatory problems converge; see, for example, Armstrong and Cardaliaguet~\cite{armstrong2015stochastic} and Feldman and Souganidis~\cite{feldman2016homogenization}. It should be noted that the counterexamples of Ziliotto~\cite{ziliotto2016stochastic} and \cite{feldman2016homogenization} yield that in the setting of nonconvex homogenization in random media is not possible to prove the existence of correctors for all directions.
\smallskip


Our main result states that a  corrector in the direction $p$ exists provided  $p$ is an extreme point of the convex hull of the sub level set $\{q\in \R^d: \bar H(q)\leq \bar H(p)\}$.  For instance, this is the case  if the law of the pair $(A,H)$ under $\P$ is radially symmetric and $A,H$ satisfy some structure conditions. 
\smallskip

This kind of result is already known in the context of first passage percolation, where the correctors are known as Buseman function;  see, for example, Licea and Newman~\cite{LN}. The techniques we use here are  strongly inspired by the arguments of Damron and Hanson~\cite{damron2014busemann}. There the authors build a type of weak solutions and prove that, when the time function is strictly convex,  they are  are actually genuine Buseman functions. 

\section{The assumptions and the main result} 

The underlying probability space is denoted by $(\Omega, \F, \P)$, where $\Omega$ is a Polish space, $\F$ is the Borel $\sigma-$field on $\Omega$ and $\P$ is a Borel probability measure. We assume that there exists a one-parameter group $(\tau_x)_{x\in \R^d}$ of measure preserving transformations on $\Omega$, that is  $\tau_x:\Omega\to \Omega$ preserves the measure $\P$ for any $x\in \R^d$ and $\tau_{x+y}= \tau_x\circ \tau_y$ for $x,y \in \Rd.$  The maps $A:(\R^d\backslash\{0\})\times \R^d\times \Omega\to \mathbb S^{d,+},$  the set of $d\times d$ real symmetric and nonnegative matrices,  and $H:\R^d\times \R^d\times \Omega\to \R$ are supposed to be continuous in all variables and stationary, that is, for all 
$p\in \R^d\backslash\{0\}, x,z\in \Rd$ and   $\omega\in \Omega$, 
$$
(A,H)(p,x,\tau_z\omega)= (A,H)(p, x+z, \omega).
$$

We also remark that equations below, unless otherwise specified,  are understood in the Crandall-Lions viscosity sense. 
\smallskip

To avoid any unnecessary assumptions, in what follows we state a general condition, which we call assumption $({\bf H})$, on the support of $\P$. 
\smallskip

{\bf Assumption $({\bf H})$:}
We assume  that, for any $p\in \R^d$, the approximate corrector equation \be\label{eq:eqdelta}
\delta v^{\delta,p} -{\rm tr}(A(Dv^{\delta,p}+p,x,\omega)D^2v^{\delta,p}) +H(Dv^{\delta,p}+p, x, \omega)=0 \quad {\rm in }\ \  \R^d,
\ee
has a comparison principle, and that, for any $R>0$, there exists $C_R>0$ such that, if  $|p|\leq R$, then the unique  solution $v^{\delta,p}$ to \eqref{eq:eqdelta} satisfies
$$
\|\delta v^{\delta,p}\|_\infty+ \|Dv^{\delta,p} \|_\infty \leq C_R. 
$$
Conditions ensuring the comparison principle are well documented; see, for instance, the Crandall, Ishii, Lions ``User's Guide''\cite{CIL}. Given the comparison principle, it is well-known that  
$$
\|v^{\delta,p}(\cdot,\omega)\|_\infty\leq \sup_{x\in \R^d} |H(0,x,\omega)|/\delta, 
$$ 
so that the $L^\infty$-assumption on $\delta v^\delta$ is not very restrictive. 
The Lipschitz bound, however, is more subtle and relies in general on a coercivity condition on the Hamiltonian. Such a structure condition is discussed, in particular,  in \cite{LS2}. 
\smallskip

Our main result is stated next.

\begin{thm}\label{main1} Assume $({\bf H})$ and, in addition,  suppose that homogenization holds in probability, that is, for any $p\in \R^d$, the family $(\delta v^{\delta,p}(0,\cdot))_{\delta>0}$ converges, as $\delta \to 0$,  in probability to some constant $-\bar H(p)$, where $\bar H:\R^d\to\R$ is a continuous and coercive map. Let $p\in \R^d$ be an extreme point of the convex hull of the sub level-set $\{q\in \R^d: \bar H(q)\leq \bar H(p)\}$. Then, for $\P-$a.e. $\omega\in \Omega$, there exists a  corrector $\chi:\R^d\times \Omega \to \Rd$ associated with $p$ and $\omega$, that is a  Lipschitz continuous  solution to \eqref{eq:correctorIntro} satisfying \eqref{cond.sublin}. 
\end{thm}

Some observations and remarks are in order here. 
\smallskip

We begin noting that that we do not know if  the corrector $\chi$ has stationary increment, and we do not expect this to be true  in general.
\smallskip

The existence of a  corrector yields that, in fact,  the $\delta
v^{\delta,p}(0,\omega)$'s converge to  $-\bar H(p)$ for $\P-$a.e. $\omega\in
\Omega$; see Proposition 1.2 in \cite{LS1}. In the rest of the paper we will use this fact repeatedly. Note also  that convexity plays absolutely no role here.
\smallskip

Our result readily applies to the case where  (${\bf
H}$) holds, the law of the pair $(A,H)$ under $\P$ is radially symmetric, and $A=A(p,x,\omega)$ and $H=H(p,x,\omega)$
are homogeneous in $p$ of degree $0$ and $1$ respectively; this is stated in Corollary \ref{cor:cor2}. Moreover, since $\bar H(p)=\bar
c|p|$ for some positive $\bar c$,  Theorem \ref{main1} implies the existence
of a corrector for any direction $p$.  Note that this case covers the homogenization
of equations of mean curvature type and the result is new. Other known results for such equations are quantitative.

\smallskip

This result  also  extends to the case where $H$ satisfies, for all $p,x \in \Rd$, $\omega \in \Omega$ and  $\lambda
\in [0,1],$ 
$$
0\leq H(\lambda p, x,\omega)\leq \lambda H(p,x,\omega).
$$
Then there exists a corrector for any direction $p$ such that $\bar H(p)$
is positive. 
Indeed, following Corollary \ref{cor:cor2}, homogenization holds in probability
for any direction $p$ and $\bar H(p)= \bar c(|p|)$ for some map $\bar c$
which is increasing when positive. 
\smallskip

If $H$ is convex in $p$ and $A$ is independent of $p$, our 
proof implies that, for any $p\in \R^d$, the limit  $\lim_{\delta \to 9}\delta v^{\delta,p}(0,\cdot)$
exists in probability; see Proposition \ref{prop:CvInProba}. This result and its the proof are very much in the flavor of \cite{LS3}. 
\smallskip

Finally  we note that our arguments  also yield the existence of a
corrector  in some directions and, thus, homogenization, 
for nonconvex Hamiltonians and $p$ dependent $A$.  More
precisely, for any direction $p$, there exists a constant $\bar c$ such that
$p$  belongs to the convex hull of directions $p'$ for  which a corrector
exists with  associated homogenized constant equal to $\bar c$; see Corollary
\ref{cor:cor1}.

\section{The Proof of Theorem~\ref{main1}}
Fix $R>0$, let $C_R$ be as in  $({\bf H})$ and  define the metric space 
$$
\Theta := \left\{\theta \in C^{0,1}(\R^d): \theta(0)=0 \; {\rm and}\; \|D\theta\|_\infty\leq C_R\right\}
$$
with distance, for all $\theta_1,\theta_2\in \Theta,$
$$
d(\theta_1,\theta_2):= \sup_{x\in \R^d} \frac{|\theta_1(x)-\theta_2(x)|}{1+|x|^{2}}.
$$
It is immediate that  $\Theta$ is a compact.

\smallskip

Next we enlarge the probability space to  $\tilde \Omega := \Omega\times \Theta\times [-C_R,C_R]$, which is endowed  with the one parameter group of transformations $\tilde \tau_x:\tilde \Omega\to \tilde \Omega$ defined, for $x\in
\R^d$, by  
$$
\tilde \tau_x (\omega,\theta,s) = (\tau_x\omega, \theta(\cdot+x)-\theta(x),s);
$$
below we abuse of notation and write  $\tilde \tau_x (\theta) = \theta(\cdot+x)-\theta(x)$.
\smallskip

Fix $p\in \R^d$ with $|p|\leq R,$ let $v^{\delta,p}$ be the solution to \eqref{eq:eqdelta}, define the map $\Phi_{\delta,p}:\Omega\to \tilde \Omega$ by  
$$
\Phi_{\delta,p}(\omega)= (\omega, v^{\delta,p}(\cdot, \omega)-v^{\delta,p}(0, \omega),-\delta v^{\delta,p}(0,\omega)),
$$
which is clearly measurable, and consider the push-forward measure 
$$
\mu_{\delta,p} = \Phi_{\delta,p} \sharp \P,
$$
which  is  a Borel probability measure on $\tilde \Omega$. 
\smallskip

Note that, since the first marginal of $\mu_{\delta,p}$ is $\P$ and  $\Omega$ is a Polish space while $\Theta\times [-C_R,C_R]$ is  compact, the family of measures $(\mu_{\delta,p})_{\delta >0}$ is tight.

\smallskip
 
Let $\mu$ be a limit, up to a subsequence $\delta_n \to 0$, of the $\mu_{\delta_n,p}$'s. 

\begin{lem}\label{lem:lemma5}
For each $x\in \R^d$, the transformation $\tilde \tau_x$ preserves the measure $\mu$.
\end{lem}

\begin{proof}
Fix a continuous and bounded map $\xi:\tilde \Omega\to \R$. Since the map $\tilde \omega\to \xi(\tau_x(\tilde \omega))$ is  continuous
and bounded and  $\mu_{\delta_n,p}$ converges weakly to $\mu$, we have 
$$
\begin{array}{rl}
\ds \int_{\tilde \Omega} \xi(\tilde \omega)\tau_x\sharp \mu(d\tilde \omega) 
 \; = & \ds
\int_{\tilde \Omega} \xi(\tau_x(\tilde \omega)) \mu(d\tilde \omega) 
=   \;
\lim_n
\int_{\tilde \Omega} \xi(\tau_x(\tilde \omega)) \mu_{\delta_n,p}(d\tilde \omega).
\end{array}
$$ 
In view of the definition of $\tilde \tau$ and $\mu_{\delta_n}$, we get 
$$
\begin{array}{rl}
 \int_{\tilde \Omega} \xi(\tau_x(\tilde \omega)) \mu_{\delta_n,p}(d\tilde \omega)
\; = & \ds
\int_{ \Omega} \xi(\tau_x\omega, v^{\delta_n,p}(x+\cdot,\omega)-v^{\delta_n,p}(x,\omega), -{\delta_n} v^{\delta_n,p}(0, \omega)) d\P(\omega) \\[1.5mm]
\; = &\ds
\int_{ \Omega} \xi(\tau_x\omega, v^{\delta_n,p}(\cdot,\tau_x\omega)-v^{\delta_n,p}(0,\tau_x\omega), -{\delta_n} v^{\delta_n,p}(-x, \tau_x\omega)) d\P(\omega) \\[1.5mm]
\; = &\ds
\int_{ \Omega} \xi(\omega, v^{\delta_n,p}(\cdot,\omega)-v^{\delta_n,p}(0,\omega), -{\delta_n} v^{\delta_n,p}(-x, \omega)) d\P(\omega),
\end{array}
$$
the last line being a consequence of the stationarity of $\P$.
\smallskip

Using that  $v^{\delta_n,p}$ is   Lipschitz continuous uniformly in $\delta$  and $\xi$ is continuous on the  set $\tilde \Omega$, we find  
$$
\begin{array}{rl}
\ds \int_{\tilde \Omega} \xi(\tau_x(\tilde \omega)) \mu_{\delta_n,p}(d\tilde \omega)
\; = & \ds
\int_{ \Omega} \xi(\omega, v^{\delta_n,p}(\cdot,\omega)-v^{\delta_n,p}(0,\omega), -{\delta_n} v^{\delta_n,p}(0, \omega)+O(\delta_n)) d\P(\omega)\\
\; = & \ds
\int_{\tilde \Omega} \xi(\tilde \omega) d\mu_{\delta_n,p}(\tilde \omega) + o(1).
\end{array}
$$
Letting $n\to+\infty$ we finally get
$$
\int_{\tilde \Omega} \xi(\tilde \omega)\tau_x\sharp \mu(d\tilde \omega) 
=
\int_{\tilde \Omega} \xi(\tilde \omega) d\mu(\tilde \omega),
$$
and, hence, the claim.
\end{proof}
The next lemma  asserts that there exists some  $\bar c=\bar c(p)$ such that  the restriction of $\mu$ to the last component is just a Dirac mass. If we know that homogenization holds, then $\bar
c(p)$ is of course nothing but $\bar H(p)$.  Note that in what follows,  abusing once again the notation, we 
denote by $\mu$ the restriction of $\mu$ to the first two components $\Omega\times
\Theta$. 
\smallskip

\begin{lem}\label{lem:lem+} There exits a constant $\bar c=\bar c(p, (\delta_n)_{n\in \N})$ such that, for any Borel measurable set $E\subset \Omega\times \Theta$, $$
\mu(E\times [-M_p,M_p])= \mu(E\times \{\bar c\}). 
$$ 
In particular, the sequence $(\delta_n v^{\delta_n,p}(0,\cdot))_{n\in
\N}$ converges in probability to $-\bar c$.
\end{lem}

\begin{proof} Let $n\geq 1$ large and   $k\in \{0, \dots, 2n\}$,  set $t_k:= -M_p+M_pk/n$ and 
$$
E_k:= \{\omega\in \Omega: \  \exists \theta \in \Theta \ \text{and} 
\  \exists s\in [t_k,t_{k+1}] \ \text{such that} \  (\omega, \theta, s)\in {\rm sppt}(\mu)\}.
$$
Since  the first marginal of $\mu$ is $\P$ and 
$
\bigcup_{l=0}^{2n} E_l\times \Theta\times [-M_p,M_p] \supset {\rm sppt}(\mu),
$
there exists  $k\in  \{0, \dots, 2n\}$ such that   $\P(E_k)>0.$
\smallskip

It turns out that $E_k$ is translation invariant, that is, for each $x\in
\R^d$,     $\tau_x E_k=E_k$. 
 Indeed, if $\omega\in \tau_xE_k$, there exists 
$\theta\in \Theta$ and $s\in [t_k,t_{k+1}]$ such that $(\tau_{-x}\omega, \theta, s)\in {\rm sppt}(\mu)$ and, hence,  $\tilde \tau_{-x}(\omega, \theta(\cdot+x)-\theta(x), s)$ belongs to ${\rm sppt}(\mu)$. Since $\mu$ is invariant under $\tilde \tau_x$, so is its support. Hence $(\omega, \theta(\cdot+x)-\theta(x), s)\in {\rm sppt}(\mu)$ and $\omega$ belongs to $E_k$. The opposite implication follows in the same way.
\smallskip 
 
The ergodicity of  $\P$ yields that $\P[E_k=1]$, which means that $\mu$ is concentrated in some $E_k\times \Theta\times [t_k,t_{k+1}]$. Thus $\mu$ is also concentrated on  $\Omega\times \Theta\times [t_k,t_{k+1}]$. Letting $n\to+\infty$ implies that there exists $\bar c\in [-M_p,M_p]$ such that $\mu$ is concentrated of the set $\Omega\times \Theta\times \{\bar c\}$. \smallskip

It remains to check that $(\delta_n v^{\delta_n,p}(0,\cdot))_{n \in \N}$ converges in probability to $-\bar c$. This is a consequence of the classical Porte-Manteau Theorem, since,  for any $\ep>0$,  
$$
\begin{array}{rl}
\ds \limsup_{n\to \infty} \P[| \delta_nv^{\delta_n,p}(0,\cdot)+\bar c|\geq \ep] \;  = & \ds \limsup \mu_{\delta_n,p} [\Omega\times \Theta\times ([-M_p,M_p]\backslash (\bar c-\ep,\bar c+\ep))] \\[1.5mm]
\leq & \ds \mu [\Omega\times \Theta\times ([-M_p,M_p]\backslash (\bar c-\ep,\bar c+\ep))] =0.
\end{array}
$$
\end{proof}
The next lemma is the first step in finding a corrector and possibly identifying $\bar c$ and $\bar H(p)$, when the latter exists. 

\begin{lem}\label{lem:lemma4} Let $\bar c$ be defined by Lemma \ref{lem:lem+}. For for  $\mu-$a.e. $(\omega, \theta)\in  \Omega\times \Theta$, $\theta$ is a  solution to 
\be\label{eq:corrector}
-{\rm tr}(A(D\theta,+p,x,\omega)D^2\theta) +H(D\theta+p, x, \omega)=\bar c \quad {\rm in }\ \ \R^d.
\ee
\end{lem}

\begin{proof} 
Fix $R,\ep>0$ and let $E(R,\ep)$ be the set of $(\omega,\theta)\in \Omega\times \Theta$ such that $\theta$ such that, in the  
open ball $B_R(0)$,
$$
 -{\rm tr}(A(D\theta+p, x,\omega)D^2\theta) +H(D\theta+p, x, \omega)\geq \bar c-\ep
 $$
 and 
$$
 -{\rm tr}(A(D\theta+p,x,\omega)D^2\theta) +H(D\theta+p, x, \omega)\leq \bar c+\ep.
 $$ 
Recall that Lemma \ref{lem:lem+} gives that  
 $(\delta_n v^{\delta_n,p}(0,\cdot))_{n\in \N}$ converge in probability to $-\bar c$. Since  $v^{\delta_n,p}$ solves \eqref{eq:eqdelta} and is uniformly Lipschitz continuous, it follows that, as $n\to \infty$,   $\mu_{\delta_n,p}(E(R,\ep))\to 1.$  
\smallskip

Finally observing that  $E(R,\ep)$ is closed in $\Omega\times \Theta$,  we infer, using again the  Porte-Manteau Theorem,  that 
$\mu(E(R,\ep))= 1$. 
\smallskip
 
As  $R$ and $\ep$ are arbitrary, we  conclude that the set $(\omega,\theta)$ for which the equation is satisfied in the viscosity sense is of full probability. 
\end{proof}
Next we investigate some properties of $\theta$. 

\begin{lem}\label{lem:lemma6} For any $x\in \R^d$, 
$
\E_\mu\left[ \theta(x)\right]=0.
$
\end{lem}

\begin{proof} Since the map $(\omega,\theta)\to \theta(x)$ is continuous on $\Omega\times \Theta$ and $v^{\delta_n,p}$ is stationary, we have 
$$
\begin{array}{rl}
\ds \E_\mu\left[ \theta(x)\right] \; =  \ds \lim \E_{\mu_{\delta_n,p}} \left[ \theta(x)\right] 
= & \ds  \lim \E_{\P} \left[ v^{\delta_n,p}(x)-v^{\delta_n,p}(0)\right] =0. 
\end{array}
$$
 
\end{proof}

\begin{lem}\label{lem:lemma7} For $\mu-$a.e. $\tilde \omega=(\omega,\theta)$ and any direction $q\in \Q^d$, the (random) limit 
$$
\rho_{\tilde \omega} (q):=\lim_{t\to \infty} \frac{\theta(tq)}{t}
$$
exists. Moreover, $\rho_{\tilde \omega}(q)$ is invariant under $\tilde \tau_x$ for $x\in \R^d$, that is, 
$$
\rho_{\tilde \tau_x(\tilde \omega)}(q)= \rho_{\tilde \omega}(q)\qquad \mu-{\rm a.e.}
$$
\end{lem}

\begin{proof} We first show that, for any $r>0$,  the limit 
$$
\lim_{t\to+\infty} \frac{1}{t} \left(\int_{B_r(0)} \theta(tq+y)dy - \int_{B_r(0)} \theta(y)dy\right) 
$$
exists $\P-$a.s. 
\smallskip

Since  the uniform
converge of uniformly Lipschitz continuous maps implies the $L^\infty$-weak $\star$ 
convergence of their gradients, the map $\xi :  \Omega\times \Theta\to \R$ defined by 
$$
\xi((\omega,\theta)): =\int_{B_r(0)} D\theta(y)\cdot q \  dy
$$ 
is continuous and bounded on $ \Omega\times \Theta$.
\smallskip

Moreover,   
$$
 \frac{1}{t} \left(\int_{B_r(0)} \theta(tq+y)dy - \int_{B_r(0)} \theta(y)dy\right) 
 =
 \frac{1}{t} \int_0^t \int_{B_r(0)} D\theta(sq+y)\cdot q \ dy ds
 = 
 \frac{1}{t} \int_0^t \xi( \tilde \tau_{sq}(\tilde \omega)) ds.
 $$
It follows from the ergodic theorem that the above expression has, as $t\to \infty$  and $\mu$-a.s. a limit $\rho_{\tilde \omega}(q,r).$
\smallskip

Choosing $r=1/n$ and letting $n\to+\infty$, we also find that, as  $t\to+\infty$,   $\theta(tq)/t$ has $\mu-$a.s. a limit $\rho_{\tilde \omega}(q)=\lim_{n\to infty}\rho_{\tilde \omega}(q,1/n)$  because $\theta$ is $C_R-$Lipschitz continuous.

Fix $x\in \R^d$ and $\tilde \omega\in \tilde \Omega$ for which $\rho_{\tilde \omega}(q)$ and $\rho_{\tilde \tau_x(\tilde\omega)}(q)$ are well defined; recall that this holds for $\mu-$a.e. $\tilde \omega$. 
\smallskip

Then, in view of the  Lipschitz continuity of $\theta$, we have  
$$
\rho_{\tilde \tau_x(\tilde \omega)}(q)= \lim_{t\to+\infty} \frac{{\tilde\tau_x(\theta)}(tq)}{t} = \lim_{t\to+\infty} \frac{1}{t} (\theta(x+tq)-\theta(x))= \rho_{\tilde \omega}(q).
$$
\end{proof}

\begin{lem}\label{lem:lemma8} There exists a random vector $\br\in L^\infty_\mu(\tilde \Omega; \R^d)$ such that, $\mu$-a.s. and for any direction $v\in \R^d$, 
$$
\lim_{t\to+\infty} \frac{\theta(tv)}{t}= \br_{\tilde \omega}\cdot v.
$$
\end{lem}

\begin{proof} Since  $\theta$ is $C_R-$Lipschitz continuous, it is enough  to check that the map $q\to \rho_{\tilde \omega}(q)$ is linear on $\Q^d$ for $\mu-$a.e. $\tilde \omega$. 
\smallskip

Let $\tilde \Omega_0$ be a set of $\mu-$full probability in $\Omega$ such that the limit $\rho_{\tilde \omega}(q)$ in Lemma \ref{lem:lemma7} exists for any $q\in \Q^d$. 
\smallskip

Restricting further the set $\Omega_0$ if necessary, we may also assume (see, for instance, the proof of Lemma 4.1 in \cite{ASo1}) that, for any $\eta,M>0$ and $\tilde \omega=(\omega,\theta)\in \Omega_0$, there exists $T>0$ such that, for all $q\in \Q^d$ with $|q|\leq M$, all $x\in \R^d$ and  $t\geq T$, 
$$
 \left| \frac{\theta(x+tq)-{\theta}(x)}{t}-\rho_{\tilde \omega}(q)\right| \leq \eta(|x|+1).
$$
Fix  $\eta,M>0$, $q_1,q_2\in \Q^d$ with $|q_1|,|q_2|\leq M$,  $\tilde \omega\in \Omega_0$ and $\eta>0$, and let $T$  be associated with $\eta,M$ as above. Then, for any $t\geq T$, we have 
$$
\theta (t(q_1+q_2))
=
\theta (t(q_1+q_2))- \theta (tq_2)+\theta (tq_2).
$$
Thus
$$
\begin{array}{l}
\ds \left|\frac{\theta (t(q_1+q_2))}{t}- \rho_{\tilde \omega}(q_1)- \rho_{\tilde \omega}(q_2)\right| \\[1.5mm]
\qquad \ds \leq 
\left| \frac{{\theta} (t(q_1+q_2))- {\theta} (tq_2)}{t} - \rho_{\tilde \omega}(q_1)\right|
+
\left| \frac{{\theta} (tq_2)}{t}-\rho_{\tilde \omega}(q_2)\right|\\[2mm]
\qquad \ds \leq 
\eta (| q_2| +t^{-1})+ \eta
\end{array}
$$
Letting  $t\to+\infty$ and $\eta\to 0$ yields the claim since $\eta$ and $M$ are arbitrary. 

\end{proof}

\begin{lem}\label{lem:lemma9} Let $\br$ be defined as in Lemma \ref{lem:lemma8}. Then $\E_\mu[\br]=0.$
\end{lem}

\begin{proof}  Lemma \ref{lem:lemma6} yields that, for any $v\in \R^d$,  
$$
0= \lim_{t\to+\infty} \E_\mu\left[ \frac{\theta (tv)}{t}\right]= \E_\mu\left[\lim_{t\to+\infty}\frac{\theta (tv)}{t}\right]= \E_\mu\left[ \br\cdot v\right]=
\E_\mu[\br]\cdot v. 
$$
\end{proof}

As a straightforward consequence of the previous results, we have the existence of a  corrector and, hence, homogenization for at least one vector $p'$. 

\begin{cor}\label{cor:cor1} For $\mu-$a.e. $\tilde \omega=(\omega,\theta,\bar c)$,  $\lim_{\delta \to 0} \delta v^{\delta,p'}(0,\omega)$ exists for $p':=p+{\bf r}_{\tilde \omega}$ and is given by $\bar c$. Moreover, $\theta'(x):=\theta(x)-{\bf r}_{\tilde \omega}\cdot x $ is a  corrector for $p'$, in the sense that 
$$
 -{\rm tr}(A(D\theta'+p',x,\omega)D^2\theta') +H(D\theta'+p', x, \omega)=\bar c \quad {\rm in } \ \  \R^d \ \ \text{with} \ \  \lim_{|x|\to+\infty} \theta'(x)/|x|=0. 
 $$
\end{cor} 

%
%

 Another consequence of the above results is that homogenization holds if the law of $(A,H)$ under $\P$ is a  radially symmetric.

\begin{cor}\label{cor:cor2}  Assume that, $\P-$a.s.,  $A=A(p,x,\omega)$ is $0-$homogeneous in $p$, $H$ satisfies,
 for all $\lambda \in [0,1],$ 
\be\label{e.hypHlambda}
0\leq H(\lambda p, x,\omega)\leq \lambda H(p,x,\omega). 
\ee
and suppose that the law of $(A,H)$ under $\P$ is radially symmetric. Then homogenization holds in probability, that is, for any $p\in \R^d$, $\lim_{\delta \to 0}-\delta  v^{\delta,p}(0,\cdot)=\bar c(|p|)$  in probability. Moreover, the map $s\to \bar c(s)$ satisfies, for any $0<s_1<s_2$,
$$
0\leq \bar c(s_1)/s_1\leq \bar c(s_2)/s_2. 
$$
\end{cor}
Note that the map $\bar c$ is increasing as soon as it is positive. Moreover, one easily checks that, if $H$ is $1-$homogeneous in $p$ and coercive, then  $\bar c(s)=\bar c s$ for some positive constant $\bar c$. 

\begin{proof} It follows from the assumed bounds and the stationarity, that there exists a set $\Omega_0$ with $\P[\Omega_0]=1$ such that, for any $p\in \R^d$ and $\omega\in \Omega_0$, $\bar c^+(p):=\limsup_{\delta \to 0}-\delta  v^{\delta,p}(0,\omega) $ and $\bar c^-(p):=\liminf_{\delta \to 0}-\delta  v^{\delta,p}(0,\omega)$
exist and are deterministic. The radial symmetry assumption and as well as  \eqref{e.hypHlambda} imply that  $\bar c^\pm(p)= \bar c^\pm( |p|)$  and, in addition, for all $\lambda \in [0,1],$ 
$$
0\leq \bar c^\pm(\lambda s)\leq \lambda \bar c^\pm (s). 
$$
Also note that the maps  $s\to \bar c^\pm(s)$ are nondecreasing. Indeed given  $0<s_1<s_2$, choosing  $s=s_2$ and $\lambda=s_1/s_2$), we find
$$
c^\pm(s_1)/s_1\leq c^\pm(s_2)/s_2\leq c^\pm(s_2)/s_1.
$$
To show that $\bar c^+=\bar c^-$,  let $p\in \R^d$, $\mu$, $\bar c$ and ${\bf r}$ be associated with $p$ as in the previous steps. For $\mu-$a.e. $\tilde \omega=(\omega, \theta,\bar c)$ with $\omega\in \Omega_0$, $\theta'(x):=\theta(x)-{\bf r}_{\tilde \omega}\cdot x$ is a  corrector for $p':=p+{\bf r}_{\tilde \omega}$ and  ergodic constant $\bar c$. It follows that $\lim_{\delta \to 0}(-\delta v^{\delta, p'}(0,\omega))=\bar c,$ 
and, hence, 
$$\bar c=\bar c^+(|p'|)=\bar c^-(|p'|).$$ 
Since $\E[p+{\bf r}]= p$,  there exist $\tilde \omega_1$ and $\tilde \omega_2$ as above such that $|p'_1|\leq |p|\leq |p'_2|$. 
\smallskip

Thus 
$$
\bar c^+(|p|)\leq \bar c^+(|p'_2|)=\bar c= \bar c^-(|p'_1|)\leq \bar c^-(|p|),
$$
and  $\bar c^+(|p|)=\bar c^-(|p|)$. 
\end{proof}

%
%

Another application of the previous results is the convergence in law of the random variable $\delta v^{\delta,p}(0,\cdot)$ when $H$ is convex in the gradient variable. The argument is a variant of \cite{LS3}. Of course, the result is much weaker than the a.s. convergence is established in \cite{LS2}; see also  \cite{ASo1, ASo3}). 
The proof is, however, rather simple.  

\begin{prop}\label{prop:CvInProba} Assume that, $\P-$a.e., $H=H(p,x,\omega)$ is convex in the $p$ variable and that $A=A(x,\omega)$ does not depend on $p$. Then, for any $p\in \R^d$, homogenization holds in probability, that is there exists $\bar H(p)$ such that $\lim_{\delta \to 0} \delta v^\delta(0, \cdot)=-\bar H(p)$ in probability.
\end{prop}

\begin{proof} Let $\mu$ be a measure built as in the beginning of the section. It follows that 
there exists a random  family of measures $\mu_\omega$ on $\Theta$ such that, for any continuous map $\phi:\Omega\times \Theta\to\R$, one has 
$$
\int_{\Omega\times \Theta} \phi(\omega,\theta) d\mu(\omega,\theta)=\int_{\Omega}\left[\int_{\Theta} \phi(\omega,\theta)d\mu_\omega(\theta)\right]d\P(\omega).
$$
Set $\hat \theta(x,\omega):=  \int_{\Theta}\theta(x)d\mu_\omega(\theta)$. Since $\P$ and $\mu$ are invariant with respect to $(\tau_z)_{z\in \Rd}$ and $(\tilde \tau_z)_{z\in \Rd}$ respectively, for any bounded measurable map $\phi=\phi(\omega)$ and any $z\in \R^d$, we have 
$$
\begin{array}{rl}
\ds \int_{\Omega} \phi(\omega)(\hat \theta(x+z,\omega) -\hat \theta(z))d\P(\omega)\; 
= & \ds 
\ds \int_{\Omega\times \Theta} \phi(\omega)(\theta(x+z) -\theta(z))d\mu(\omega,\theta)
\\[1.5mm]
= \; \ds 
 \ds \int_{\Omega\times \Theta} \phi(\tau_{-z}\omega)\theta(x)\tilde \tau_{z}\sharp d\mu(\omega,\theta)
= & \ds 
 \ds \int_{\Omega} \phi(\tau_{-z}\omega)\hat \theta(x,\omega) d\P(\omega)
=\ds \int_{\Omega} \phi(\omega)\hat \theta(x,\tau_z\omega) d\P(\omega).
\end{array}
$$
This shows that $\hat \theta$ has stationary increments. Moreover, in view of  Lemma \ref{lem:lemma6}, $\hat \theta$ has mean zero, and, hence,  $D\hat \theta$ is stationary with average $0$. In particular, $\hat \theta$ is $\P-$a.s. strictly sublinear at infinity. Since, for $\mu-$a.e. $(\omega,\theta)$, $\theta$ is a solution to \eqref{eq:corrector} and $H$ is convex in the gradient variable, $\hat \theta$ is a subsolution to \eqref{eq:corrector} and, thus 
a subcorrector. Following \cite{LS3}, this implies that 
$$
\liminf_{\delta\to 0} \delta v^{\delta,p}(0,\omega) \geq -\bar c. 
$$
In particular, for any  sequence $(\delta_n')_{n'\in \N}$ which tends to $0$ such that $(\mu_{\delta_n',p})_{n'\in \N}$ and 

$(\delta_n' v^{\delta_n',p}(0))_{n'\in \N}$ converge respectively to a measure $\mu'$ and a constant $-\bar c'$, we have   $\bar c'\leq \bar c$. Exchanging the roles of $(\delta_n)_{n\in \N}$ and $(\delta_n')_{n'\in \N}$ leads to the equality $\bar c=\bar c'$. The conclusion now follows.

\end{proof}

We are now ready to prove our main result. 

\begin{proof}[Proof of Theorem \ref{main1}]  We assume that homogenization holds in probability and  $p\in \R^d$ is  an extreme point of the convex hull of the set $S:=\{q\in \R^d: \bar H(q)\leq \bar H(p)\}$. 
\smallskip

Let $\mu$ be a measure built as in the beginning of the section and ${\bf r}$ be defined by Lemma \ref{lem:lemma8}. 
\small

Then $\bar H(p+{\bf r})=\bar H(p)$ $\mu-$a.s., that is  $p+{\bf r}$ belongs to $S$ $\mu-$a.s. Indeed Lemma \ref{lem:lemma4} gives  $\bar c=\bar H(p)$ and 
$$
-{\rm tr}(A(D\theta+p, x,\omega)D^2\theta) +H(D\theta+p, x, \omega)=\bar c \ \  {\rm in } \ \ \R^d,
$$
while,in view of Lemma \ref{lem:lemma8},  for all $x\in \R^d$, 
$$
\lim_{t\to+\infty} \frac{\theta(tx)}{t}={\bf r}\cdot x. 
$$
Thus $\tilde \theta(x):= \theta(x)-{\bf r}\cdot x$ is a  corrector for $p+ {\bf r}$, that is  it  satisfies 
$$
-{\rm tr}(A(D\tilde \theta+p+{\bf r},x,\omega)D^2\tilde \theta) +H(D\tilde \theta+p+{\bf r}, x, \omega)=\bar c \quad {\rm in } \ \ \R^d \ \ \text{and} \ \  \lim_{|x|\to+\infty} \tilde \theta(x)/|x|=0.
$$
It follows that  $\bar H(p+{\bf r})=\bar H(p)$ $\mu-$a.s.. 
\smallskip

Next we recall (Lemma \ref{lem:lemma9}) that  $\E_\mu[p+{\bf r}]=p$. Since $p+{\bf r} \in S$ $\mu-$a.s. and $p$ is an extreme point of the convex hull of $S$, the equality $\E_\mu[p+{\bf r}]=p$ implies that ${\bf r}=0$ $\mu-$a.s..  Therefore 
$\ds
\lim_{|x|\to+\infty} \theta(x)/|x|=0
$
$\mu-$a.s., 
which, together with the fact that $\theta$ solves the corrector equation for $p$, implies that $\theta$ is a  corrector for $p$ itself. 
\end{proof}


\bibliographystyle{plain}

\newcommand{\noop}[1]{} \def\cprime{$'$} \def\cprime{$'$}

\end{document}